\documentclass[12pt]{amsart}
\usepackage{amssymb,amsmath}
\newtheorem{proposition}{Proposition}[section]
\newtheorem{theorem}[proposition]{Theorem}
\newtheorem{corollary}[proposition]{Corollary}
\newtheorem{lemma}[proposition]{Lemma}
\theoremstyle{definition}
\newtheorem{definition}[proposition]{Definition}
\newtheorem{remark}[proposition]{Remark}

\numberwithin{equation}{section}

\def \x{\mathbf x}
\title[Transmission problem for nonlinear plate]
{A nonlinear transmission problem \\for a compound plate with thermoelastic part}

\author[M. Potomkin]
{M. Potomkin}

\address{Kharkov National University - Department of Mathematics and Mechanics
\newline\indent
4 Svobody sq, 61077 Kharkiv, Ukraine}
\email{mika\_potemkin@mail.ru}


\begin{document}
\begin{abstract}
In this paper we study a nonlinear transmission problem for a plate which consists of thermoelastic and isothermal parts. The problem generates a dynamical system in a suitable Hilbert space. Main result is the proof of the asymptotic smoothness of this dynamical system. Also we prove the existence of a compact global attractor in particular cases when the nonlinearity is of Berger type or scalar.      
\end{abstract}

\maketitle 

\section {Introduction}
In this paper we deal with a partially thermoelastic plate: one part is of isothermal material, the second one is of material which structure does not allow neglecting thermal dissipation. Due to thermal dissipation purely thermoelastic plate is exponentially stable in linear case (see, e.g., survey in \cite[Chapter 3A]{LasieTriggi}) or possesses a compact global attractor in cases of different kind of nonlinearities (see, e.g., \cite{ChuiBucci,Book,Manuscript} for Berger type of nonlinearity and \cite{ChuiLasie,chla26,Manuscript,chla31,Karman,khan} for von Karman type). From the other hand, in the case of purely isothermal plate the energy is constant thus there could not be any decay to zero point in linear model and global attractor in nonlinear model. Here we investigate whether the thermal dissipation on a part of the plate is enough for a plate to have any stabilization. Exponential stability of linear problem of this type has been established in \cite{mr}.    

Let $\Omega_1$, $\Omega_2$ and $\Omega$ are bounded open sets in $\mathbb R^2$, $\Gamma_0=\overline{\Omega}_1\cap\overline{\Omega}_2$, $\Gamma_1=\partial\Omega_1/\Gamma_0$ and $\Gamma_2=\partial\Omega_2/\Gamma_0$ are smooth surfaces. Also we set $\Omega=\Omega_1\cup\Omega_2\cup\Gamma_0$ and assume that $\overline{\Gamma}_1\cap\overline{\Gamma}_2=\emptyset$. 
In the model under consideration the plate (its middle surface), in equilibrium, occupies the domain $\Omega$ which consists of two parts $\Omega_1$ and $\Omega_2$ with common boundary $\Gamma_0$. In what follows below $\nu$ denotes the outward vector on $\Gamma_1$ and $\Gamma_2$, in cases of common boundary $\Gamma_0$ the vector $\nu$ is outward for $\Omega_2$.  
       
Functions $u(\x,t)$ and $v(\x,t)$ denotes the vertical displacement of the plate, the function $\theta(\x,t)$ responds to the temperature regime. 

The equations are as follows 
\begin{eqnarray}
\rho_1 u_{tt}+\beta_1\Delta^2 u + \mu\Delta \theta +F_1(u,v) =0 \;\;&&\mathrm{in}\;\;\Omega_1\times \mathbb R^+\label{1} \\
\rho_0 \theta_t -\beta_0 \Delta \theta -\mu\Delta u_t=0\;\;&&\mathrm{in}\;\;\Omega_1\times \mathbb R^+\label {2} \\
\rho_2 v_{tt}+\beta_2\Delta^2 v +F_2(u,v)=0\;\;&&\mathrm{in}\;\;\Omega_2\times \mathbb R^+ \label{3}
\end{eqnarray}
Boundary conditions imposed on $u$ and $v$ along $\Gamma_1$ and $\Gamma_2$ are clamped
\begin{equation}\label{4}
u=\frac{\partial u}{\partial \nu}=0\;\; \mathrm{on}\;\; \Gamma_1\times\mathbb R^+,\;\;\;\;
v=\frac{\partial v}{\partial \nu}=0\;\; \mathrm{on}\;\; \Gamma_2\times\mathbb R^+
\end{equation}  
We assume that $\theta$ satisfies Newton law of cooling (with coefficient $\lambda\geq0$) through the $\Gamma_1$ and $\theta$ vanishes on $\Gamma_0$
\begin{equation}
\theta=0 \;\; \mathrm{on}\;\; \Gamma_0\times\mathbb R^+,\;\;\;\;
\frac{\partial \theta}{\partial \nu}+\lambda\theta=0\;\; \mathrm{on}\;\; \Gamma_1\times\mathbb R^+\label{5}
\end{equation}
Also we impose the following boundary conditions on $\Gamma_0$
\begin{equation}
u=v,\;\;\frac{\partial u}{\partial \nu}=\frac{\partial v}{\partial \nu}, \;\;
\beta_1\Delta u=\beta_2\Delta v,\;\; \beta_1\frac{\partial \Delta u}{\partial \nu}+\mu\frac{\partial \theta}{\partial \nu}=\beta_2\frac{\partial \Delta v}{\partial \nu}\;\;\mathrm{on}\;\;\Gamma_0\times\mathbb R^+.\label{6}
\end{equation}
Initial data:
\begin{equation} \label {7}
\begin{array}{l}
u(\x,0)=u_0(\x),\;\;u_t(\x,0)=u_1(\x),\;\;\theta(\x,0)=\theta_0(\x)\;\;\mathrm{on}\;\;\Omega_1,\\ 
v(\x,0)=v_0(\x),\;\;v_t(\x,0)=v_1(\x)\;\;\mathrm{on}\;\;\Omega_2.
\end{array}
\end{equation}

Coefficients  $\rho_i$, $\beta_i$ and $\mu$ are strictly positive, $F_i\;:H^2(\Omega_1)\times H^2(\Omega_2)\longrightarrow L^2(\Omega_i)$ are nonlinear functions (we impose more conditions on $F_i$ in the next section when all necessary notations will be introduced).    

In this paper we propose general approach for nonlinear transmission plate problems in study of the property of asymptotic smoothness (its definition is given in Subsection 4.1). We may apply our abstract result to at least three concrete problems. 

{\bf Problem A} describes oscillations of a plate in Berger approach. In this case 
$$
F_1(u,v)=-M(u,v) \Delta u,\;\;\;F_2(u,v)=-M(u,v)\Delta v,
$$  
where 
\begin{equation}\label{Bergern}
M(u,v)=\Gamma+\gamma\left[\int\limits_{\Omega_1}|\nabla u|^2\mathrm{d}\x+\int\limits_{\Omega_2}|\nabla v|^2\mathrm{d}\x\right],
\end{equation}
here $\Gamma$ is a real number, $\gamma$ is strictly positive. 

In {\bf problem B} we consider scalar nonlinearities, namely, 
$$
F_1(u,v)=f_1(u),\;\;\;F_2(u,v)=f_2(v),
$$  
where scalar functions $f_i\in C^2$ satisfy 
\begin{eqnarray}
&&|f'_i(s)|\leq C(1+|s|^p),\;\;\exists p>1, C>0,\\
&&\liminf\limits_{|s|\rightarrow\infty}\frac{f_i(s)}{s}>0.
\end{eqnarray}

{\bf Problem C} deals with von Karman nonlinearity. Here we set $\Gamma_2=\emptyset$ and 
$$
F_1(u,v)=-[u,\mathcal{F}_1],\;\;\;F_2(u,v)=-[v,\mathcal{F}_2],
$$  
where $[\psi,\varphi]=\psi_{xx}\varphi_{yy}+\psi_{yy}\varphi_{xx}-2\psi_{xy}\varphi_{xy}$ is the von Karman brackets; the Airy stress functions $\mathcal{F}_1$ and $\mathcal{F}_2$ solve (parameters $\gamma_i$ are strictly positive)
\begin{equation}\label{karman}
\gamma_1\Delta^2\mathcal{F}_1+[u,u]=0\;\mathrm{in}\;\Omega_1\times{\mathbb R^+}\;\;\mathrm{and}\;\;\gamma_2\Delta^2\mathcal{F}_2+[v,v]=0\;\mathrm{in}\;\Omega_2\times{\mathbb R^+}
\end{equation}
with boundary conditions 
$$
\begin{array}{ll}
\mathrm{on} \;\Gamma_1:&\mathcal{F}_1=\frac{\partial}{\partial\nu}\mathcal{F}_1=0,\\
\mathrm{on}\; \Gamma_0:&\mathcal{F}_1=\mathcal{F}_2,\;\;
\frac{\partial}{\partial\nu}\mathcal{F}_1=\frac{\partial}{\partial\nu}\mathcal{F}_2,\;\;
\gamma_1\Delta\mathcal{F}_1=\gamma_2\Delta\mathcal{F}_2,\;\;
\gamma_1\frac{\partial}{\partial\nu}\Delta\mathcal{F}_1=\gamma_2\frac{\partial}{\partial\nu}\Delta\mathcal{F}_2.
\end{array}
$$

Analysis of linear thermoelastic plates with different boundary conditions is given in \cite[Chapter 3]{LasieTriggi}. Also we refer to earlier works \cite{Avalos,Avalos2}. See also \cite{Lagnese}.  

Plate models with nonlinear term of type \eqref{Bergern} (we call such plates/models to be of Berger type) were introduced in \cite{Berger}. Asymptotic behavior of isothermal Berger plate with mechanical dissipation were analyzed in \cite[Chapter 4]{Book} and, as particular case, thermoelastic Berger plate were investigated in \cite{ChuiBucci}, relative results are obtained in \cite{Giorgi}. All works contain the proof of the existence of a compact global attractor. Models with scalar nonlinearities and different type of damping were considered for wave equations. We refer to \cite{Manuscript} where plenty of such models were considered in the framework of the method of stabilizability estimate and relative estimates which we use in the paper (see \eqref{th_ineq} and \eqref{stab_ineq} here). Many results devoted to von Karman equation are collected in \cite{Karman}. See also earlier works \cite{ChuiLasie,chla26,chla31,khan}.    

For the examples of investigation of transmission problems similar to the problem in this paper we refer to \cite{cr,marnari,mr,rn} and references therein. Asymptotic behavior of composite systems with localized damping is considered in \cite{Tynd,ChTynd} (for the problem \eqref{1}-\eqref{6} we may say that there is a localized thermal damping in equations). Also we refer to \cite{Lada} which contains the consideration of elliptic transmission problem and to \cite{Lions} where transmission problem for Navier-Stokes equation is considered.        

Key tool in the investigation of long-time behavior is appropriate estimates on the norm (of the phase space) of the difference of two solutions with initial data from a positively invariant set (the observability estimates). We follow the methods described in \cite{Manuscript,Karman} (see also \cite{ChuiBucci,Tynd,ChuiLasie,chla27,chla26,chla31,ChTynd,khan} and references therein) where such estimates allow to prove the existence of a global attractor, its smoothness and finite-dimensionality; these estimates may differ from model to model and their derivation strongly depends on considered model. Also for one model it is useful to consider two such estimates when one of them is weaker then second one but requires less conditions. For example, for problem B the estimate \eqref{phoh} gives just asymptotic smoothness; when we additionally set $f_2\equiv0$ then \eqref{phoh} could be transformed into \eqref{stab_ineq} which gives smoothness and finite-dimensionality.  
          
The most difficult part in derivation of observability estimate concerns with linear part of equation. In Theorem \ref{aux_the} we give the inequality which holds for any nonlinearity; this result could be useful if one consider the type of nonlinearity different from that is under consideration in this paper. In the proof we use multiplicators applied earlier, e.g., in \cite{Tynd, mr} (for origin works we refer to \cite{Avalos, Avalos2, Lagnese} and references therein).     
 
Our first main result is the property of asymptotic smoothness. To achieve it we use method of  so-called compensated compactness function first introduced in \cite{khan}. We apply this method using formulation in \cite[Proposition 2.10]{Manuscript}. As in \cite{mr} we need to impose conditions on papameters:
\begin{equation}\label{imposion}
\rho_1\geq \rho_2 \;\;\mathrm{and}\;\;\beta_1\leq\beta_2.
\end{equation}
and geometric structure of $\Omega$
\begin{eqnarray}
&&(\x-\x_0)\cdot\nu(\x)\geq \delta_0\;\mathrm{on}\;\Gamma_0,\label{geom1}\\
&&(\x-\x_0)\cdot\nu(\x)\leq 0\;\mathrm{on}\;\Gamma_2\label{geom2}
\end{eqnarray} 
for some $\x_0\in\mathbb R^2$ and $\delta_0>0$.

The exact definition of asymptotic smoothness is given in Section 4, now we just say that this property implies the existence of compact attractors for every bounded positevily invariant set. With the help of idea of a stabilizability estimate (see \cite{Manuscript}) for problem B we prove that these "local" attractors are smooth and of finite fractal dimension. For problems A and B (for problem B with the restriction $f_2\equiv 0$, as in \cite{marnari}) we prove the existence of a compact global attractor. This is our second main result. The main difficulty how to obtain global (not just local) attractor lies in the proof that appropriate Lyapunov function (mechanical energy of the system) is constant only on stationary trajectories. 
    
Up to our best knowledge asymptotic behavior in transmission problem for a plate of type A, B and C was not considered before. 
 
The paper is organized as follows: in the next section we will establish well-posedness of the problem; in Section 3 we establish the main auxiliary inequality. Subsection 4.1 contains necessary definitions and statements we use. The proof of asymptotic smoothness is given in Subsection 4.2. The proof of the existence of compact global attractor is given in Subsection 4.3.   

\section{Preliminaries and well-posedness}
We introduce notations similar to introduced in \cite{mr}. In the sequel if a function $\psi(\x)$ defined for $\x \in \Omega$ is equal to $\psi_i(\x)$ if ${\x}\in\Omega_i$, $i=1,2$, then we denote $\psi=\left\{\psi_1,\psi_2\right\}$. 

The following Hilbert spaces will be used in the sequel 
$$
H^2_{T}:=\left\{\left.\left\{\phi_1,\phi_2\right\}\in H^2(\Omega_1)\times H^2(\Omega_2)\right|
\begin{array}{l} \phi_i=\frac{\partial \phi_i}{\partial \nu}=0 \;\;\mathrm{on}\;\;\Gamma_i,\\
\phi_1=\phi_2 \;\;\mathrm{and} \;\; \frac{\partial\phi_1}{\partial \nu}=\frac{\partial \phi_2}{\partial\nu}\;\;\mathrm{on}\;\;\Gamma_0\end{array}\right\}
$$
$$
H^1_{D}:=\left\{\phi\in H^1(\Omega_1)\;:\;\phi=0 \;\;\mathrm{on}\;\;\Gamma_0\right\}
$$  
with inner products 
$$
\begin{array}{rll}
(\left\{w_1,w_2\right\},\left\{\phi_1,\phi_2\right\})_{H^2_T}
&:=&\int_{\Omega_1}\beta_1\Delta w_1\Delta \phi_1\mathrm{d}\x+\int_{\Omega_2}\beta_2\Delta w_2\Delta\phi_2\mathrm{d}\x,\\ 
(w,\phi)_{H^1_D}&:=&\int_{\Omega_1}\beta_0\nabla w\cdot\nabla\phi\mathrm{d}\x+\int_{\Gamma_1} \beta_0 \lambda w \phi\mathrm{d}\x.
\end{array}
$$
In what follows $H^{-2}_T$ and $H^{-1}_{D}$ denote the dual spaces of $H^2_T$ and $H^1_D$, respectively. Also we need to note that $H^2_T=H^2_0(\Omega)$ and their norms are equivalent. Thus there exists $C>0$ such that 
$$
||u||_{H^2(\Omega_1)}+||v||_{H^2(\Omega_2)}\leq C (||\Delta u||_{L^2(\Omega_1)}+||\Delta v||_{L^2(\Omega_2)})
$$  
for all $\left\{u,v\right\}\in H^2_T$.

In order to formulate well-posedness result and use linear semigroup theory we rewrite the original system as a Cauchy problem. To this end we consider the operators (following \cite{mr}, see also  \cite{Lagnese})
$$\begin{array}{c}
A_0\;:\;H^2_T\rightarrow H^{-2}_{T},\;\;\;
A_1\;:\;L^2(\Omega_1)\times L^2(\Omega_2)\times L^2(\Omega_1)\rightarrow L^2(\Omega_1)\times L^2(\Omega_2)\times L^2(\Omega_1),\\
B_0\;:\;H^2_T\times H^1_D\rightarrow H^{-2}_{T}\times H^{-1}_{D}
\end{array}
$$ 
given by 
$$
\begin{array}{rcl}
\left\langle A_0\left\{w_1,w_2\right\},\left\{\phi_1,\phi_2\right\}\right\rangle
&:=&(\left\{w_1,w_2\right\},\left\{\phi_1,\phi_2\right\})_{H^2_T},\\ 
A_1\left\{w_1,w_2,w_3\right\}&:=&\left\{\rho_1 w_1, \rho_2 w_2, \rho_0 w_3\right\} \\ 
\left\langle B_0\left\{w_1,w_2,w_3\right\},\left\{\phi_1,\phi_2,\phi_3\right\}\right\rangle
&:=&\int_{\Omega_1} \mu (w_3\Delta \phi_1-\Delta w_1\phi_3)\mathrm{d}\x + (w_3,\phi_3)_{H_{D}^1},
\end{array}
$$
Here a trilinear form $\left\{\cdot,\cdot,\cdot\right\}$ denotes a vector from a corresponding subspace of $L^2(\Omega_1)\times L^2(\Omega_2)\times L^2(\Omega_1)$. Also we denote $B_1\left\{\phi_1,\phi_2\right\}:= \left\{A_0\left\{\phi_1,\phi_2\right\},0\right\}$ and
$$
\mathbb{A}=\left(
\begin{array}{cc}
I_1&0\\
0&A_1
\end{array}\right),\;\;\;\;\;
\mathbb{B}=\left(
\begin{array}{cc}
0& -I_2\\
B_1& B_0
\end{array}\right)
$$
Here $I_1\left\{w_1,w_2\right\}:=\left\{w_1,w_2\right\}$ and $I_2\left\{w_1,w_2,w_3\right\}:=\left\{w_1,w_2\right\}$. We note that linear part of original equations \eqref{1}-\eqref{6} (i.e., when $F_1=F_2=0$) could be rewritten in terms of introduced operators as follows 
$$
{\mathbb A}\frac{\mathrm{d}w}{\mathrm{d}t}+{\mathbb B}w=0,\;\;w(t)=(u(t),v(t),u_t(t),v_t(t),\theta(t)),\;\;t>0. 
$$

Finally, let us introduce the linear operator $\Lambda:=-\mathbb{A}^{-1}\mathbb{B}$ with the domain $D(\Lambda)$ in $\mathcal{H}=H^2_T\times L^2(\Omega_1)\times L^2(\Omega_2)\times L^2(\Omega_1)$ given by 
$$
D(\Lambda):= \left\{
\begin{array}{l}
w=(w_1,w_2,w_3,w_4,w_5)\in H^2_T\times H^2_T\times H^1_D\;\;\mathrm{such}\;\;\mathrm{that}\\
\hspace{40 pt}
\left\{A_0\left\{w_1,w_2\right\},0\right\}+B_0\left\{w_3,w_4,w_5\right\}\in L^2(\Omega_1)\times L^2(\Omega_2)\times L^2(\Omega_1)
\end{array}
\right\}
$$

Repeating arguments of \cite[Lemma 3.10]{mr} one can assure that there holds 
$$
D(\Lambda)=D_0\equiv \left\{
\begin{array}{l}
w\in \left[H^2_{T}\cap\left(H^4(\Omega_1)\times H^4(\Omega_2)\right)\right]\times H^2_T\times\left[H^2(\Omega_1)\cap H_D^1\right]:\\\hspace{20 pt}
\beta_1\Delta w_1=\beta_2\Delta w_2 \;\mathrm{and}\;\beta_1\frac{\partial\Delta w_1}{\partial\nu}+\mu\frac{\partial\theta}{\partial \nu}=\beta_2\frac{\partial\Delta w_2}{\partial\nu} \;\mathrm{on}\;\Gamma_0,\\\hspace{80 pt}\;\frac{\partial w_5}{\partial \nu}+\lambda w_5=0\;\mathrm{on}\;\Gamma_1 
\end{array}\right\}.
$$ 

\begin{theorem}\label{th1}
The operator $\Lambda$ is a generator of a strongly continuous semigroup of contractions on~$\mathcal{H}$.
\end{theorem}

\begin{proof}
See \cite[Theorem 2.3]{mr}.\\
\end{proof}  

Now let us denote for any $w=(w_1,w_2,w_3,w_4,w_5)\in \mathcal{H}$  
$$
\mathbb{F}(w)=\mathbb {A}^{-1}(0,0,-F_1(w_1,w_2), -F_2(w_1,w_2), 0).
$$

We impose the following conditions
\begin{eqnarray}\label{lipsitz}
&&\int_{\Omega_1}|F_1(w^1_1,w^1_2)-F_1(w^2_1,w^2_2)|^2\mathrm{d}\x+\int_{\Omega_2}|F_2(w^1_1,w^1_2)-F_2(w^2_1,w_2^2)|^2\mathrm{d}\x\\\nonumber&&\hspace{180 pt}\leq C(r)||\left\{w^1_1-w^2_1,w^1_2-w^2_2\right\}||^2_{H^2_T}
\end{eqnarray}
for all $||\left\{w^i_1,w^i_2\right\}||_{H^2_T}\leq r,\;i=1,2$. Also we assume that there exists such continuous functional $\Pi:H^2_{0}(\Omega)\rightarrow{\mathbb R}$ that 
\begin{eqnarray}
\label{Pi_1} && \frac{\mathrm{d}}{\mathrm{d}t}\Pi(w_1,w_2)=\int_{\Omega_1}F_1(w_1,w_2)w_{1,t}\mathrm{d}\x+\int_{\Omega_2}F_2(w_1,w_2)w_{2,t}\mathrm{d}\x,\\
&& \label{Pi_2} \Pi(w_1,w_2)\geq -C,\;\;\exists C>0, \\
&& \label{Pi_3} \Pi(w_1,w_2)\leq {\mathcal G}\left(||\left\{w_1,w_2\right\}||_{H^2_0(\Omega)}\right). 
\end{eqnarray}

The condition \eqref{Pi_1} holds for $\left\{w_1,w_2\right\}\in L^2(0,T;H^2_{0}(\Omega))$:$\left\{w_{1,t},w_{2,t}\right\}\in L^2((0,T)\times~\Omega)$; the conditions \eqref{Pi_2} and \eqref{Pi_3} hold for all $\left\{w_1,w_2\right\}\in H^2_{0}(\Omega)$. The scalar function $\mathcal{G}\;:\;\mathbb{R}^+\longrightarrow\mathbb{R}^+$ is supposed to be bounded on bounded intervals. The condition \eqref{Pi_1} also means that feedback forces $\left\{F_1(u,v),F_2(u,v)\right\}$ are potential, i.e., $\left\{F_1(u,v),F_2(u,v)\right\}$ is a Frechet derivative of $\Pi(u,v)$ (we denote it by $\Pi'_{\Phi}(u,v))$).  

For problem A $\Pi(w_1,w_2)=\frac{1}{4}M^2(w_1,w_2)$, for problem B 
$$\Pi(w_1,w_2)=\int_{\Omega_1}\int_{0}^{w_1(x)}f_1(s)\mathrm{d}s\mathrm{d}{\x}+\int_{\Omega_1}\int_{0}^{w_2(x)}f_2(s)\mathrm{d}s\mathrm{d}{\x}.$$

And for problem C if calculate $\mathcal{F}_i$ according to \eqref{karman}:
\begin{equation}\label{Pifk}
\Pi(w_1,w_2)=\frac{\gamma_1}{2}\int_{\Omega_1}|\Delta\mathcal{F}_1|^2\mathrm{d}\x+\frac{\gamma_2}{2}\int_{\Omega_2}|\Delta\mathcal{F}_2|^2\mathrm{d}\x.
\end{equation}

We denote by $H^{s}\equiv H^s(\Omega)$ a Hilbert space equipped with the norm $||\cdot||_s$. 

For von Karman case we need the following auxiliary result (which is the part of \cite[Proposition 6.1]{Manuscript} if $\gamma_1=\gamma_2$): 
\begin{lemma} \label{lfk}
Let $\mathcal{F}(w)$ be a solution to the problem \eqref{karman} for some $w=\left\{u,v\right\}$. Next statements hold true
\begin{eqnarray}
&||\left[w_1,w_2\right]||_{-2}\leq C||w_1||_{2-\beta}||w_2||_{1+\beta}\;\;\; \beta\in[0,1),\label{kb1}\\
&||\left[w_1,w_2\right]||_{-1-2\delta}\leq C||w_1||_{2-\delta}||w_2||_{2-\delta}\;\;\; \delta\in (0,1/2),\label{kb2}\\
&||\mathcal{F}(w_1)-\mathcal{F}(w_2)||_{2}\leq C||w_1+w_2||_{2-\beta}\cdot||w_1-w_2||_{1+\beta}\;\;\; \beta\in[0,1).\label{kb4}
\end{eqnarray}
\end{lemma} 
\begin{proof}
The inequalities \eqref{kb1} and \eqref{kb2} are the same as (6.6) in \cite[Proposition 6.1]{Manuscript} with $j=2$ and (6.7) in \cite[Proposition 6.1]{Manuscript} (see also \cite[Subsection 1.4]{Karman}), respectively.
To prove \eqref{kb4} we need elliptic regularity for transmission problem formulated in \cite[Lemma 2.2a)]{mr} (see also \cite{Lada}): 
$$
||\mathcal{F}(w_1)-\mathcal{F}(w_2)||_{2}\leq C||\left[w_1+w_2,w_1-w_2\right]||_{-2}.
$$  
The application of \eqref{kb1} gives \eqref{kb4}. 
\end{proof} 
Lemma \ref{lfk} implies the following corollary (see proof of \cite[Lemma 6.2]{Manuscript}), which is necessary if apply the result about asymptotic smoothness to von Karman problem. 
\begin{corollary}
Consider any $w_1,w_2\in H^2_0(\Omega)$ such that $||w_i||_2\leq r$ and $\Pi$ calculated by \eqref{Pifk}. There exists such $\delta\in(0,1/2)$ that 
\begin{eqnarray}
&||\Pi'_{\Phi}(w_1)-\Pi'_{\Phi}(w_2)||\leq C(r)||w_1-w_2||_{2},\nonumber\\
&|\Pi(w_1)-\Pi(w_2)|\leq C(r)||w_1-w_2||_{2-\delta},\nonumber\\
&||\Pi'_{\Phi}(w_1)-\Pi'_{\Phi}(w_2)||_{-1-2\delta}\leq C(r)||w_1-w_2||_{2-\delta}.\nonumber
\end{eqnarray}
\end{corollary} 
Setting $w=(u,v,u_t,v_t,\theta)$ the origin equations \eqref{1}-\eqref{7} could be rewritten as the following Cauchy problem 
\begin{equation} \label{cauchy}
\frac{\mathrm{d}w}{\mathrm{d}t}=\Lambda w +\mathbb{F}(w),\;\;w|_{t=0}=w_0.
\end{equation}

Now we are in position to give a definition of a mild (according to \cite[Chapter 6]{Pazy}) solution to the problem \eqref{1}-\eqref{7}: 
\begin{definition} 
Consider any $T>0$. A solution $w\in C([0,T];\mathcal{H})$ of the integral equation
\begin{equation}\label {Diamel}
w(t)=e^{t\Lambda}w_0+\int\limits_{0}^{t}e^{(t-\tau)\Lambda}\mathbb{F}(w(\tau))\mathrm{d}t 
\end{equation}
is called a mild solution to the problem \eqref{1}-\eqref{7} on interval $[0,T]$ with initial condition $w(0)=w_0$.
\end{definition}

For future use we need energy functional (or Lyapunov function) $\mathcal{E}\;:\;\mathcal{H}\longrightarrow \mathbb{R}$ defined for an argument $w=(w_1,w_2,w_3,w_4,w_5)$ as follows
\begin{equation}\label{Lyap}
\mathcal{E}(w)=\frac{1}{2}\left[\int_{\Omega_1}\beta_1|\Delta w_1|^2+\rho_1|w_3|^2+\rho_0|w_5|^2\mathrm{d}\x+\int_{\Omega_2}\beta_2|\Delta w_2|^2+\rho_2|w_4|^2\mathrm{d}\x+2\Pi(w_1,w_2)\right]. 
\end{equation}

The well-posedness result is given by 
\begin{theorem}\label{wp}
Let \eqref{lipsitz},\eqref{Pi_1},\eqref{Pi_2} and \eqref{Pi_3} hold.
Next statements hold true:
\begin{list}{}{}
\item[{(i)}] For any initial $w_0\in \mathcal{H}$ and $T>0$ there exists a unique mild solution $w(t)\in C([0,T];\mathcal{H})$. Moreover, it satisfies energy equality
\begin{equation}\label{Eniq}
\mathcal{E}(w(T))-\mathcal{E}(w(t))=-\int_{t}^{T}\int_{\Omega_1}\beta_0|\nabla w_5|^2\mathrm{d}\x\mathrm{d}\tau-
\int_{t}^{T}\int_{\Gamma_1}\beta_0\lambda|w_5|^2\mathrm{d}\Gamma\mathrm{d}\tau
\end{equation}
for all $0\leq t\leq T$.
\item[{(ii)}] If $w^{1}_{0},w^{2}_{0}\in\mathcal{H}$ and $||w^{i}_{0}||_{\mathcal{H}}\leq R$ then there exists a constant $C_{R,T}>0$ such that 
\begin{equation}\label{cont}
||w^1(t)-w^2(t)||_{\mathcal{H}}\leq C_{R,T}||w^{1}_{0}-w^{2}_{0}||_{\mathcal{H}},\;\;t\in[0,T],
\end{equation}
where $w^1(t)$ and $w^2(t)$ are mild solutions with initial data $w^1_0$ and $w^2_0$, respectively.
\item[{(iii)}] if $w_0\in D(\Lambda)$ then the corresponding mild solution is strong, i.e. it is continuously differentiable, its values lie in $D(\Lambda)$ and it satisfies the equation 
$\frac{\mathrm{d}}{\mathrm{d}t}w=\Lambda w+\mathbb{F}(w)$ in $\mathcal{H}$ for almost all $t \in [0,T]$.
\end{list}  
\end{theorem}
\begin{proof}
The existence of a local mild solution and its uniqueness follows from \cite[Theorem 6.1.4]{Pazy} (it is justified because, first, following Theorem \ref{th1}, $\Lambda$ is an infinitisimal operator of s.c. semigroup and, second, $\mathbb{F}$ is locally Lipschitz). We refer to \cite[Step II, proof of Theorem 2.1]{Potomkin} how to obtain that the solution could be extended on arbitrary long interval and the statement that for all $R>0$ there exists such $C_{R,T}$ that the following inequality holds
\begin{equation}\label{bounded}
||w(t)||_{\mathcal{H}}\leq C_{R,T}||w_0||_{\mathcal{H}}
\end{equation} 
for all $t\in [0,T]$.

To prove \eqref{cont} we use \eqref{Diamel}. Consider $\forall T>0,\;\;\forall t \in (0,T)$ and two mild solutions $w^1(t)$ and $w^2(t)$ with initial data $w^1$ and $w^2$, respectively. Assume also that $||w^i||_{\mathcal{H}}\leq R$. Then   
$$
||w^1(t)-w^2(t)||_{\mathcal{H}}\leq ||e^{t\Lambda}(w^1_0-w^2_0)||_{\mathcal{H}}
+\int_{0}^{t}||e^{(t-\tau)\Lambda}(\mathbb{F}(w^1(\tau))-\mathbb{F}(w^2(\tau)))||_{\mathcal{H}}\mathrm{d}\tau
$$
Using that $||e^{t\Lambda}||_{[\mathcal{H},\mathcal{H}]}\leq 1$, the inequality \eqref{bounded} and local Lipschitz property of $\mathcal{F}$ we obtain 
$$
||w^1(t)-w^2(t)||_{\mathcal{H}}\leq||w^1_0-w^2_0||_{\mathcal{H}}+C_{R,T}\int\limits_{0}^{t}||w^1(\tau)-w^2(\tau)||_{\mathcal{H}}\mathrm{d}\tau
$$  
Gronwall lemma gives \eqref{cont}. 

Last statement of the theorem (about strong solutions) follows directly from \cite[Theorem 6.1.5]{Pazy}. 
\end{proof}

\begin{remark} 
We say that a triple $(u,v,\theta)$ is a weak (variation) solution of \eqref{1}-\eqref{7} when 
$$
\begin{array}{c}
\left\{u,v\right\}\in L^{\infty}(0,T;H^2_T),\;\;\;\left\{u_t,v_t\right\}\in L^{\infty}(0,T;L^2(\Omega_1)\times L^2 (\Omega_2)) \\ 
\theta \in L^{\infty}(0,T;L^2(\Omega_1))\cap L^2(0,T;H^1_D)
\end{array}
$$ 
such that $u(\x,0)=u_0(\x)$, $v(\x,0)=v_0(\x)$ and  
\begin{eqnarray}
&& \int\limits_{0}^{T}\int_{\Omega_1}-\rho_1u_t\phi_{1,t}+\beta_1\Delta u\Delta \phi_1 +\mu \theta \Delta \phi_1+F_1(u,v) \phi_1\mathrm{d}\x\mathrm{d}t+\nonumber\\
&&+ \int\limits_{0}^{T}\int_{\Omega_2}-\rho_2 v_t \phi_{2,t} +\beta_2 \Delta v \Delta\phi_2+ F_2(u,v) \phi_2 \mathrm{d}\x\mathrm{d}t-\nonumber\\&&-\int_{\Omega_1}\rho_1 u_1\phi_1(0)\mathrm{d}\x-\int_{\Omega_2} \rho_2 v_1 \phi_2(0)\mathrm{d}\x+\nonumber\\
&&+ \int\limits_{0}^{T}\int_{\Omega_1}(-\rho_0\theta+\mu\Delta u)\phi_{3,t}+\beta_0\nabla \theta \nabla\phi_3\mathrm{d}\x\mathrm{d}t+\int\limits_{0}^{T}\int_{\Gamma_1}\beta_0\lambda \theta \phi_3\mathrm{d}\Gamma\mathrm{d}t-\nonumber\\&&-\int_{\Omega_1}(\rho_0\theta_0-\mu\Delta u_0)\phi_3(0)\mathrm{d}\x=0 \label{v2}
\end{eqnarray}
for all $\left\{\phi_1,\phi_2\right\}\in C^1(0,T;H^2_T)$ and $\phi_3\in C^1(0,T;H^1_D)$ such that 
$$
\phi_1(T)=\phi_2(T)=\phi_3(T)=0.
$$     

Any mild solution is a weak variation solution. To show this let us consider any initial $w_0=(u_0,v_0,u_1,v_1,\theta_0)\in \mathcal{H}$ and the corresponding mild solution $w(t)$. 
 Since $D(\Lambda)$ is dense in $\mathcal{H}$ we may choose a sequence $\left\{w_n\right\}\subset D(\Lambda)$ such that $w_n\rightarrow w_0$ in $\mathcal{H}$ and following \eqref{cont} we have $w_n(t)\rightarrow w(t)$ in $C([0,T];\mathcal{H})$. Equality \eqref{v2} holds for the corresponding strong solutions, besides there exists a constant $C_T$ such 
$$
\int\limits_{0}^{T}\int_{\Omega_1}\beta_0|\nabla \theta^n|^2\mathrm{d}\x\mathrm{d}t +
\int\limits_{0}^{T}\int_{\Gamma_1}\beta_0\lambda|\theta^n|^2\mathrm{d}\Gamma\mathrm{d}t\leq C_T.
$$   
It gives weak convergences $\nabla\theta^n\rightharpoonup\nabla\theta$ in $L^2([0,T]\times\Omega_1)$ and $\theta ^n\rightharpoonup \theta$ in $L^2([0,T]\times \Gamma_1)$. Now we may pass to limit $n\rightarrow\infty$ in \eqref{v2}. 
\end{remark}

\section {Key inequality}
In this section we obtain important auxiliary inequalities for the following system
\begin{eqnarray}
&&\rho_1 u_{tt}+\beta_1\Delta^2 u +\mu \Delta \theta=g_1(t,\x) \label {a1}\\
&&\rho_0 \theta_t-\beta_0\Delta \theta -\mu\Delta u_t =0\label{a2}\\ 
&&\rho_2 v_{tt}+\beta_2\Delta^2 v= g_2 (t,\x)\label{a3}
\end{eqnarray} 
with boundary conditions \eqref{4}-\eqref{6}. Here $g_i=g_i(t,\x)$ are functions from $L^2((0,T)\times\Omega_i)$ (below we write $g_i$ or $g_i(t)$).

Note that in the first theorem (Theorem \ref{aux_the}) we don't impose any conditions on $g_i$. All inequalities in this section hold on strong solutions.

Also we denote 
\begin{eqnarray}
\nonumber&&E(t)=\frac{1}{2}\int_{\Omega_1}\rho_1|u_t|^2+\beta_1|\Delta u|^2+\rho_0|\theta|^2\mathrm{d}\x\\\label{energy}&&\hspace{60 pt}+\frac{1}{2}
\int_{\Omega_2}\rho_2|v_t|^2+\beta_2|\Delta v|^2\mathrm{d}\x.
\end{eqnarray}

We introduce four auxiliary functionals and estimate their time derivative with solutions to \eqref{a1}-\eqref{a3} and \eqref{4}-\eqref{6} substituted. All functionals are typical for thermoelastic problems (see \cite{Avalos,Avalos2,ChuiBucci,Tynd,ChuiLasie,Karman,ChTynd,Giorgi,Lagnese,mr}; also in similar way as below the idea of scalar multiplicators $\phi_i$ and $\psi$ is applied in \cite{Tynd,mr}, without them all functionals had already been considered in \cite{Lagnese} for different plate models). Except some modifications in $J_1$ and $J_2$ our calculations are close to \cite[Section ~3]{mr}.  

As in \cite{mr} we use the following scalar functions 
$$
\phi_i(\x)=\left\{
\begin{array}{ll}
0, & \x\in U_{i\delta}(\Gamma_0)\cap\Omega_1,\\
1, & \x\in \Omega_1\setminus U_{2i\delta}(\Gamma_0)
\end{array}\right.,\;\;\phi_{i}\in C^2(\Omega_1),\;\;i=1,2, 
$$
and    
$$
\psi(\x)=\left\{
\begin{array}{ll}
1, & \x\in U_{4\delta}(\Omega_2),\\
0, & \x\in \Omega_1\setminus U_{8\delta}(\Omega_2)
\end{array}\right.,\;\;\psi\in C^2(\Omega). 
$$
Also we define $U_i=\left\{\x\in\Omega_1| \phi_i(\x)=1\right\}$ and $V=\left\{\x\in\Omega|\psi(\x)=1\right\}$. Besides, we use below a vector field $h=(h_1,h_2)\in C^2(\Omega)$ such that $h(\x)=-\nu(\x)$ if ${\x}\in\Gamma_1\cup\Gamma_2$ and a vector field $m(\x)=\x-\x_0$ where $\x_0$ is introduced in Introduction and it satisfies \eqref{geom1} and \eqref{geom2}. 

Main result of this subsection is formulated as follows 
\begin{theorem}\label{aux_the}
Let \eqref{imposion}, \eqref{geom1} and \eqref{geom2} hold.  
There exits such functional $R(t)$ and positive constants $k_0,C_0>0$ that 
\begin{equation}\label{Rb}
|R(t)|\leq C_0E(t)
\end{equation} 
and
\begin{eqnarray}
&&\nonumber\frac{\mathrm{d}}{\mathrm{d}t}R(t)\leq -k_0 E(t)+
C_0\left[\int_{\Omega_1}|\nabla\theta|^2\mathrm{d}\x+\int_{\Omega}|\left\{u,v\right\}|^2+|\Delta^{-1}_{D}\left\{\rho_1u_t,\rho_2v_t\right\}|^2\mathrm{d}\x\right]+\\
&&\label{Rmain0}\hspace{90 pt}+\int_{\Omega_1}g_1(t)\sigma_1(t)\mathrm{d}\x+\int_{\Omega_2}g_2(t)\sigma_2(t)\mathrm{d}\x
\end{eqnarray}
where 
$$\begin{array}{l} \left\{\sigma_1,\sigma_2\right\}=-q_1\Delta_D^{-1}\left\{\rho_0\phi_1\theta,0\right\}+q_2\left\{h\cdot\nabla u,h\cdot\nabla v\right\}+\\\hspace{80 pt}+q_3\left\{\phi_2u,0\right\}+q_4\left\{\psi m\cdot\nabla u,\psi m \cdot \nabla v\right\}
\end{array}$$
and $q_1,q_2,q_3$ and $q_4$ are positive numbers that will be specified later. 
\end{theorem}

In order to obtain asymptotic smoothness in the Section 4 we need the following corollary.
\begin{corollary}\label{cor_1} 
Let \eqref{imposion}, \eqref{geom1} and \eqref{geom2} hold.
If 
\begin{equation}\label{g_bnd}
\int_{\Omega_1}|g_1|^2\mathrm{d}\x+\int_{\Omega_2}|g_2|^2\mathrm{d}\x\leq C E(t)
\end{equation}
then there exists $k,C>0$
\begin{equation}
\label{Rmain}\frac{\mathrm{d}}{\mathrm{d}t}R(t)\leq -k E(t)+
C\left[\int_{\Omega_1}|\nabla\theta|^2\mathrm{d}\x+\int_{\Omega}|\left\{u,v\right\}|^2+|\Delta^{-1}_{D}\left\{\rho_1u_t,\rho_2v_t\right\}|^2\mathrm{d}\x\right]
\end{equation}
where $R(t)$ is the same as in Theorem \ref{aux_the}.
\end{corollary}

To prove Theorem \ref{aux_the} we introduce four functionals $J_i,\;i=1,2,3,4$ and then we set 
$R=\sum\limits_{i=1}^{4}q_iJ_i$.
 
\noindent{\bf The functional $J_1$.} Let $\left\{w_1,w_2\right\}:=\Delta^{-1}_{D}\left\{\rho_0\phi_1\theta,0\right\}$ where $\Delta^{-1}_{D}$ is an inverse Laplace operator with Dirichlet boundary conditions on $\partial\Omega=\Gamma_1\cup\Gamma_2$. 
\begin{equation}
J_1:=-\int_{\Omega_1}\rho_1u_tw_1\mathrm{d}\x-\int_{\Omega_2}\rho_2v_tw_2\mathrm{d}\x 
\end{equation} 
Its derivative looks as follows 
\begin{eqnarray}
\nonumber&&\frac{\mathrm{d}}{\mathrm{d}t} J_1(t)=
\int_{\Omega_1}(\beta_1\Delta^2u+\mu\Delta \theta)w_1\mathrm{d}\x-\int_{\Omega_1}g_1(t)w_1\mathrm{d}\x
+\\
\nonumber&&\hspace{40 pt}+\int_{\Omega_2}(\beta_2\Delta^2v)w_2\mathrm{d}\x-\int_{\Omega_1}g_2(t)w_2\mathrm{d}\x
\\
\nonumber&&\hspace{40 pt}-\int_{\Omega}\left\{\rho_1u_t,\rho_2v_t\right\}\Delta^{-1}_{D}\left\{\phi_1\Delta(\beta_0\theta+\mu u_t),0\right\}\mathrm{d}\x
\end{eqnarray}
Using obvious equality $\phi_1\Delta \xi=\Delta (\phi_1 \xi)-\Delta \phi_1 \xi-2\nabla \phi_1\cdot\nabla \xi$ for $\xi=\beta_0\theta+\mu u_t$ we continue computations  
\begin{eqnarray}
\nonumber&&\frac{\mathrm{d}}{\mathrm{d}t} J_1(t)
=\int_{\Omega_1}\beta_1\Delta u\Delta w_1- \mu\nabla \theta\cdot\nabla w_1\mathrm{d}\x-\int_{\Gamma_1}\beta_1\Delta u\frac{\partial w_1}{\partial\nu}\mathrm{d}\Gamma-\int_{\Omega_1}g_1(t)w_1\mathrm{d}\x
\\\nonumber&&\hspace{40 pt}+
\int_{\Omega_2}\beta_2\Delta v\Delta w_2\mathrm{d}\x-\int_{\Gamma_2}\beta_2\Delta v\frac{\partial w_2}{\partial\nu}\mathrm{d}\Gamma-\int_{\Omega_2}g_2(t)w_2\mathrm{d}\x+\\
\nonumber&&\hspace{40 pt}
-\rho_1\beta_0\int_{\Omega_1}\phi_1u_t\theta\mathrm{d}\x
-\rho_1\mu\int_{\Omega_1}\phi_1|u_t|^2\mathrm{d}\x+
\\\nonumber&&\hspace{40 pt}+\int_{\Omega}\Delta^{-1}_{D}\left\{\rho_1u_t,\rho_2v_t\right\}\left\{\Delta\phi_1(\beta_0\theta+\mu u_t),0\right\}\mathrm{d}\x- 
\\\nonumber&&\hspace{40 pt}
-2\int_{\Omega}\left\{\beta_0\theta+\mu u_t,0\right\}\left\{\Delta \phi_1,0\right\}
\Delta^{-1}_{D}\left\{\rho_1u_t,\rho_1v_t\right\}\mathrm{d}\x-
\\\nonumber&&\hspace{40 pt}
-2\int_{\Omega}\left\{\beta_0\theta+\mu u_t,0 \right\}\left(\nabla\left\{\phi_1,0\right\}\cdot\nabla\Delta^{-1}_{D}\left\{\rho_1 u_t,\rho_2 v_t\right\}\right)\mathrm{d}\x.
\end{eqnarray} 
We mention two useful estimates. 
First, 
$$
||\left\{w_1,w_2\right\}||^2_{H^2(\Omega)}+\int_{\Gamma_1}\left|\frac{\partial w_1}{\partial\nu}\right|^2\mathrm{d}\Gamma+\int_{\Gamma_2}\left|\frac{\partial w_2}{\partial \nu}\right|^2\mathrm{d}\Gamma
\leq C\int_{\Omega_1}|\theta|^2\mathrm{d}\x
$$
Second, for any $\eta_0>0$
$$
\int_{\Omega}|\nabla\Delta^{-1}_{D}\left\{\rho_1u_t,\rho_2v_t\right\}|^2\mathrm{d}\x\leq
\eta_0\int_{\Omega}|\left\{u_t,v_t\right\}|^2\mathrm{d}\x+C_{\eta_0}\int_{\Omega}|\Delta^{-1}_{D}\left\{\rho_1u_t,\rho_2v_t\right\}|^2\mathrm{d}\x
$$ 
Finally, we have 
\begin{lemma} \label {auxl1} There holds 
\begin{eqnarray}
&&\frac{\mathrm{d}}{\mathrm{d}t}J_1(t)\leq 
\eta\left\{\int_{\Gamma_1}|\Delta u|^2\mathrm{d}\Gamma+\int_{\Gamma_2}|\Delta v|^2\mathrm{d}\Gamma
\right\}+
\nonumber \\&&\hspace{40 pt}
+\eta\left\{\int_{\Omega_1}|\Delta u|^2+|u_t|^2\mathrm{d}\x+\int_{\Omega_2}|\Delta v|^2+|v_t|^2\mathrm{d}\x\right\}+\nonumber \\
&&\hspace{40 pt}+C_{\eta}\left[\int_{\Omega_1}|\nabla \theta|^2\mathrm{d}\x+\int_{\Omega}|\Delta^{-1}_{D}\left\{\rho_1u_t,\rho_2v_t\right\}|^2\mathrm{d}\x\right]- \mu\rho_1\int_{U_1}|u_t|^2\mathrm{d}\x-\nonumber
\\&&\hspace{40 pt} -\int_{\Omega_1}g_1(t)w_1\mathrm{d}\x-\int_{\Omega_2}g_2(t)w_2\mathrm{d}\x,\label{j_1}
\end{eqnarray}
where $\left\{w_1,w_2\right\}:=\Delta^{-1}_{D}\left\{\rho_0\phi_1\theta,0\right\}$.
\end{lemma}
\noindent{\bf The functional $J_2$.}
Let us consider a vector field $h=(h_1,h_2)\in [C^2(\Omega)]^2$ such that $
h(\x)=
-\nu(\x)$ if $\x\in\Gamma_1\cup\Gamma_2$
then we introduce 
$$
J_2:=\int_{\Omega_1}\rho_1 u_th\cdot\nabla u\mathrm{d}\x +\int_{\Omega_2}\rho_2v_th\cdot\nabla v\mathrm{d}\x
$$
Following the same procedure as in \cite[Lemma 3.5]{mr} we have 
\begin{lemma} \label{auxl2}
There holds 
\begin{eqnarray}
&&\frac{\mathrm{d}}{\mathrm{d}t}J_2\leq
-\frac{\beta_1}{2}\int_{\Gamma_1}|\Delta u|^2\mathrm{d}\Gamma-\frac{\beta_2}{2}\int_{\Gamma_2}|\Delta v|^2\mathrm{d}\Gamma-
\int_{\Omega_1}g_1(t)h\cdot\nabla u\mathrm{d}\x-\int_{\Omega_2}g_2(t)h\cdot\nabla v\mathrm{d}\x\nonumber\\
&&\hspace{40 pt}+
C\left\{\int_{\Omega_1}|\Delta u|^2+|u_t|^2+|\nabla\theta|^2\mathrm{d}\x+\int_{\Omega_2}|\Delta v|^2+|v_t|^2\mathrm{d}\x\right\}.\label{j_2}
\end{eqnarray}
\end{lemma}
\noindent{\bf The functional $J_3$.}
Let us introduce 
$$
J_3(t)=\int_{\Omega_1}\rho_1 u_t \phi_2 u\mathrm{d}\x
$$
We have 
\begin{lemma}\label{auxl3}
There holds 
\begin{eqnarray}
&&\frac{\mathrm{d}}{\mathrm{d}t}J_3(t)\leq 
-\frac{\beta_1}{2}\int_{U_2}|\Delta u|^2\mathrm{d}x+\rho_1 \int_{\Omega_1}|u_t|^2\mathrm{d}\x-\int_{\Omega_1}g_1(t)\phi_2u\mathrm{d}\x+\nonumber\\
&&\hspace{40 pt}+
\eta\left\{\int_{\Omega_1}|\Delta u|^2\mathrm{d}\x+\int_{\Omega_2}|\Delta v|^2\mathrm{d}\x\right\}+
C_{\eta}\left\{\int_{\Omega_1}|u|^2+|\nabla \theta|^2\mathrm{d}\x+\int_{\Omega_2}|v|^2\mathrm{d}\x\right\}.\label{j_3}
\end{eqnarray}
\end{lemma}
\noindent{\bf The functional $J_4$.}
We recall that $m(\x)=\x-\x_0$ and it satisfies \eqref{geom1} and \eqref{geom2}. We introduce 
$$
J_4=\int_{\Omega_1}\rho_1 u_t \psi m\cdot\nabla u\mathrm{d}\x+\int_{\Omega_2}\rho_2v_t\psi m\cdot\nabla v \mathrm{d}\x
$$
repeating the similar calculations as in \cite[Lemma 3.7]{mr} we obtain 
\begin{lemma}\label{auxl4}
There holds 
\begin{eqnarray}
&&\frac{\mathrm{d}}{\mathrm{d}t}J_4(t)\leq 
\eta\left(\int_{\Omega_1}|\Delta u|^2\mathrm{d}\x+\int_{\Omega_2}|\Delta v|^2\mathrm{d}\x\right)+
C_{\eta}\left(\int_{\Omega_1}|u|^2+|\nabla \theta|^2\mathrm{d}\x+\int_{\Omega_2}|v|^2\mathrm{d}\x\right)-\nonumber\\
&&\hspace{40 pt}-\frac{\rho_1-\rho_2}{2}\int_{\Gamma_0}m\cdot\nu|u_t|^2\mathrm{d}\Gamma-
\frac{\beta_2-\beta_1}{2}\int_{\Gamma_0}m\cdot\nu|\Delta u|^2\mathrm{d}\Gamma-\nonumber\\&&\hspace{40 pt}-\int_{\Omega_1}g_1(t)\psi m\cdot\nabla u\mathrm{d}\x-\int_{\Omega_2}g_2(t)\psi m\cdot\nabla v \mathrm{d}\x +\nonumber\\  
&&\hspace{40 pt}+C \int_{U_2}|u_t|^2+|\Delta u|^2\mathrm{d}\x-{\beta_1}\int_{V\cap\Omega_1}|\Delta u|^2+|u_t|^2\mathrm{d}\x
-\int_{\Omega_2}\rho_2|v_t|^2+\beta_2|\Delta v|^2\mathrm{d}\x.\label{j_4}
\end{eqnarray}
\end{lemma}
Finally, let $$
R=\sum\limits_{i=1}^{4}q_iJ_i=J_1+\frac{\eta}{\min\left\{\beta_1,\beta_2\right\}}J_2+\left(\frac{\mu}{2}-\eta C\right)J_3+\eta^{1/2}J_4
$$
with sufficiently small $\eta$.

\section{Main result: asymptotic behavior}
\subsection{Preliminary definitions and assertions}
We recall some definitions and statements (following \cite{Book,Manuscript,Raugel}) that will be needed in the sequel. All formulations are made for abstract dynamical system $(X,S_t)$ where $X$ is a complete metric space and $S_t$ is a semigroup of operators in $X$.
\begin{definition} The dynamical system $(X,S_t)$ is said to be asymptotically smooth if for any positively invariant bounded set $D\subset X$ there exists a compact $K$ in the closure $\overline{D}$ of $D$ such that 
$$
\lim\limits_{t\rightarrow +\infty}\sup\limits_{x\in D}\mathrm{dist}_X \left(S_tx,K\right)=0.
$$
\end{definition}

To prove asymptotic smoothness we use the following criterium (see \cite[Proposition 2.10]{Manuscript}) 
\begin{proposition}\label{fromManu}
Let $(X,S_t)$ be a dynamical system on a complete metric space $X$ endowed with metric $d$. Assume that for any bounded positively invariant set $B$ in $X$ and for any $\varepsilon>0$ there exists $T\equiv T(\varepsilon, B)$ such that 
\begin{equation}\label{th_ineq}
d(S_Ty_1,S_Ty_2)\leq \varepsilon +\Psi_{\varepsilon,B,T}(y_1,y_2),\;\;\;y_i\in B ,
\end{equation}
where $\Psi_{\varepsilon,B,T}(y_1,y_2)$ is a nonnegative function defined on $B\times B$ such that for all sequence $\left\{y_n\right\}$ from $B$ there exists a subsequence $\left\{y_{n_k}\right\}$
that
\begin{equation}\label{th_Psi}
\lim\limits_{k\rightarrow\infty}\lim\limits_{l\rightarrow\infty}\Psi_{\varepsilon,B,T}(y_{n_k},y_{n_l})=0.
\end{equation}
Then $(X,S_t)$ is an asymptotically smooth dynamical system.
\end{proposition}

\begin{definition}
$\mathcal{A}\subset X$  is called an attractor if (i) $\mathcal{A}$ is closed bounded strictly invariant set ($S_t\mathcal{A}=\mathcal{A}\;\forall t\geq0$) and (ii) $\mathcal{A}$ possesses the uniform attraction property, i.e. for any bounded set $B\subset X$ the following equality holds true 
$$
\lim\limits_{t\rightarrow +\infty}\sup\limits_{x\in B}\mathrm{dist}_X \left(S_tx,\mathcal{A}\right)=0.
$$
\end{definition}



For any bounded $\mathcal{B}\subset X$ we define the unstable manifold $\mathcal{M}^{u}(\mathcal{B})$ emanating from the set $\mathcal{B}$ as a set of all $y\in X$ such that there exists a full trajectory $\gamma=\left\{u(t)\;:\;t\in{\mathbb R}\right\}$ with the properties 
$$
u(0)=y\;\;\;\mathrm{and}\;\;\;\lim\limits_{t\rightarrow-\infty}\mathrm{dist}_{X}(u(t),\mathcal{B})=0.
$$ 

The following result could be found in \cite[Theorem 2.30]{Manuscript}.

\begin{theorem} Let $(X,S_t)$ be a asymptotically smooth dynamical system in a Banach space $X$. Assume that there exists a Lyapunov function $\mathcal{E}(x)$ for $(X,S_t)$ on $X$ such that $\mathcal{E}(x)$ is bounded from above on any bounded subset of $X$ and the set $\mathcal{E}_R=\left\{x:\mathcal{E}(x)\leq R\right\}$ is bounded for every $R$. Let $\mathcal{B}$ be the set of elements $x \in X$ such that there exists a full trajectory $\left\{u(t):t\in \mathbb {R}\right\}$ with the properties $u(0)=x$ and $\mathcal{E}(u(t))=\mathcal{E}(x)$ for all $t \in \mathbb {R}$. If $\mathcal{B}$ is bounded, then $(X,S_t)$ possesses a compact global attractor $\mathcal{A}=\mathcal{M}^{u}(\mathcal{B})$.    
\label{31}\end{theorem}

In what follows below in this section we deal with dynamical system $(\mathcal{H},S_t)$ where $\mathcal {H}=H^2_T\times L^2(\Omega_1)\times L^2(\Omega_2)\times L^2(\Omega_1)$ and for all $w_0\in\mathcal{H}$ we set $S_tw_0:=w(t)$ where $w(t)$ is a mild solution to the problem \eqref{1}-\eqref{6}. 
 
\subsection {Asymptotic smoothness} 
Below we denote by $H^{s},\;s\in \mathbb {R}$ a Hilbert space $H^s(\Omega)$ (e.g., for von Karman nonlinearity) or $D((-\Delta_D)^{s/2})$ equipped with the norm $||\cdot||_{s}$ where $\Delta_D$ is a Laplace operator with Dirichlet boundary conditions on $\partial \Omega$ (this operator has already been in use in Section 3). In this subsection we prove the following result: 
\begin{theorem} \label{assmooth}
Let \eqref{imposion},\eqref{geom1}, \eqref{geom2}, \eqref{lipsitz}, \eqref{Pi_2} and \eqref{Pi_3} hold. Besides, there exist such $\delta, \sigma>0$ that 
\begin{eqnarray}  
& w\longmapsto\Pi(w)&:\;H^{2-\delta}\longrightarrow{\mathbb R} \;\;\mathrm{is}\;\;\mathrm{continuous}\;\;\mathrm{mapping,}\label{Pi_4}\\
&w \longmapsto\Pi'_{\Phi}(w)&:\;H^{2-\delta}\longrightarrow H^{-\sigma} \;\;\mathrm{is}\;\;\mathrm{continuous}\;\;\mathrm{mapping.}\label{Pi_5}
\end{eqnarray}
Here $w$ is a function defined on $\Omega$ and we recall that $\Pi'_{\Phi}(w)=\left\{F_1(u,v),F_2(u,v)\right\}$ if $w=\left\{u,v\right\}$.

Then the dynamical system $(\mathcal{H},S_t)$ generated by the problem \eqref{1}-\eqref{6} (the dynamical system $(\mathcal{H},S_t)$ is introduced in the end of Subsection 4.1) is asymptotically smooth.   
\end{theorem}

In order to apply Proposition \ref{fromManu} in what follows we obtain estimate \eqref{th_ineq} for our problem. We follow the line of arguments in \cite{ChuiLasie,Manuscript, Karman}.  

For this let $(u^1(t),v^1(t),\theta^1(t))$ and $(u^2(t),v^2(t),\theta^2(t))$ be solutions to the problem \eqref{1}-\eqref{6} and assume that for any $t>0$ there exists $R>0$ such that  
\begin{equation}\label{51}
\int_{\Omega_1}\rho_1|u^{i}_t|^2+\beta_1|\Delta u^i|^2+\rho_0|\theta^i|^2\mathrm{d}\x+
\int_{\Omega_2}\rho_2|v^{i}_t|^2+\beta_2|\Delta v^i|^2\mathrm{d}\x\leq R^2 
\end{equation}
All inequalities below in this subsection are obtained for $(u(t),v(t),\theta(t))$ where
$$u(t)=u^1(t)-u^2(t),\;\;v(t)=v^1(t)-v^2(t),\;\;\theta(t)=\theta^1(t)-\theta^2(t)$$ 
This triple satisfies the system \eqref{a1}-\eqref{a3} with $g_1(t)=F_1(u^2,v^2)-F_1(u^1,v^1)$ and $g_2(t)=F_2(u^2,v^2)-F_2(u^1,v^1)$. Also this triple satisfies boundary conditions \eqref{4}-\eqref{6}.  
First, Corollary \ref{cor_1} obviously implies  
\begin{lemma}
Let \eqref{imposion}, \eqref{geom1} and \eqref{geom2} hold. There exists $C>0$ such that 
\begin{eqnarray}
&&\int_{t}^{T}E(t)\mathrm{d}t\leq
C(E(t)+E(T))+C\left[\int_{t}^{T}\int_{\Omega_1}|u|^2+|\nabla \theta|^2\mathrm{d}\x\mathrm{d}t\right]+\nonumber\\
&&\hspace{55 pt}+C\left[
\int_{t}^{T}\int_{\Omega_2}|v|^2\mathrm{d}\x\mathrm{d}t
+\int_{t}^{T}\int_{\Omega}|\Delta^{-1}_{D}\left\{\rho_1u_t,\rho_2v_t\right\}|^2\mathrm{d}\x\mathrm{d}t\right] \label{nain}
\end{eqnarray}
for all $t,T:\; 0<t\leq T$.
\end{lemma}
\begin {proposition}
There exists $T_0>0,$ $C>0$ such that for $T>T_0$
\begin{eqnarray}
&&TE(T)+\int_{0}^{T}E(t)\mathrm{d}t\leq C\left\{\int_{0}^{T}\int_{\Omega_1}|\nabla\theta|^2\mathrm{d}\x\mathrm{d}t+\Psi_T(u,v)\right\}+\nonumber\\&&
\hspace{40 pt} +C\left\{\int_{0}^{T}\int_{\Omega_1}|u|^2\mathrm{d}\x\mathrm{d}t+\int_{0}^{T}\int_{\Omega_2}|v|^2\mathrm{d}\x\mathrm{d}t+\int_{0}^{T}\int_{\Omega}|\Delta^{-1}_{D}\left\{\rho_1u_t,\rho_2v_t\right\}|^2\mathrm{d}\x\mathrm{d}t\right\}\label{as_key}
\end{eqnarray} 
where 
\begin{eqnarray}
&&\Psi_T(u,v):=\left|\int_{0}^{T}\int_{\Omega_1}g_1(t)u_t\mathrm{d}\x\mathrm{d}t+\int_{0}^{T}\int_{\Omega_2}g_2(t)v_t\mathrm{d}\x\mathrm{d}t\right|+\nonumber \\
&&\hspace{30 pt}+
\left|\int_{0}^{T}\int_{t}^{T}\int_{\Omega_1}g_1(\tau)u_t\mathrm{d}\x\mathrm{d}\tau\mathrm{d}t+\int_{0}^{T}\int_{t}^{T}\int_{\Omega_2}g_2(\tau)v_t\mathrm{d}\x\mathrm{d}\tau\mathrm{d}t\right|\label{psipsi}
\end{eqnarray}
\end{proposition}
\begin{proof}
Energy indentity implies 
\begin{eqnarray}\nonumber
&&E(0)=E(T)+\beta_0\int_{0}^{T}\int_{\Omega_1}|\nabla\theta|^2\mathrm{d}\x\mathrm{d}t+\beta_0\lambda\int_{0}^{T}\int_{\Gamma_1}|\theta|^2\mathrm{d}\Gamma\mathrm{d}t -\\\nonumber &&
\hspace{80 pt}-\int_{0}^{T}\int_{\Omega_1}g_1u_t\mathrm{d}\x\mathrm{d}t-\int_{0}^{T}\int_{\Omega_2}g_2v_t\mathrm{d}\x\mathrm{d}t.
\end{eqnarray}
Then using \eqref{nain} we have
\begin{eqnarray}
&&\int_{0}^{T}E(t)\mathrm{d}t\leq 
CE(T)+
C\left|\int_{0}^{T}\int_{\Omega_1}g_1u_t\mathrm{d}\x\mathrm{d}t+\int_{0}^{T}\int_{\Omega_2}g_2v_t\mathrm{d}\x\mathrm{d}t\right|+\nonumber\\ 
&&\hspace{60 pt}+C\int_{0}^{T} \int_{\Omega_1}|\nabla \theta|^2\mathrm{d}\x\mathrm{d}t+
C\left\{\int_{0}^{T}\int_{\Omega_1}|u|^2\mathrm{d}\x+\int_{0}^{T}\int_{\Omega_2}|v|^2\mathrm{d}\x\mathrm{d}t\right\}+\nonumber\\ 
&&\hspace{60pt} +C\int_{0}^{T}\int_{\Omega}|\Delta^{-1}_{D}\left\{\rho_1u_t,\rho_2v_t\right\}|^2\mathrm{d}\x\mathrm{d}t\label{psihelp}
\end{eqnarray} 
Energy identity also implies 
$$
E(T)\leq E(t)+\int_{t}^{T}\int_{\Omega_1}g_1u_t\mathrm{d}\x\mathrm{d}\tau+\int_{t}^{T}\int_{\Omega_2}g_2v_t\mathrm{d}\x\mathrm{d}\tau
$$
Integration over $t\in[0,T]$ gives 
\begin{equation}\nonumber
TE(T)\leq \int_{0}^{T}E(t)\mathrm{d}t+\left|\int_{0}^{T}\int_{t}^{T}\int_{\Omega_1}g_1u_t\mathrm{d}\x\mathrm{d}\tau\mathrm{d}t+\int_{0}^{T}\int_{t}^{T}\int_{\Omega_2}g_2v_t\mathrm{d}\x\mathrm{d}\tau\mathrm{d}t\right|
\end{equation}
Taking $T>3C$ and adding \eqref{psihelp} multiplied by $\frac{3}{2}$ to the last inequality we obtain \eqref{as_key}.  
\end{proof}
Next result of this subsection is 
\begin{theorem}
For given $\varepsilon$ and $T>T_0$ there holds 
\begin{eqnarray}
&&E(T)\leq \varepsilon+ C_T\Psi_T(u,v)+\nonumber\\
&&\hspace{30 pt}+ C_T\left[\int_{0}^{T}\int_{\Omega_1}|u|^2\mathrm{d}\x\mathrm{d}t+\int_{0}^{T}\int_{\Omega_2}|v|^2\mathrm{d}\x\mathrm{d}t+\int_{0}^{T}\int_{\Omega}|\Delta^{-1}_{D}\left\{\rho_1u_t,\rho_2v_t\right\}|^2\mathrm{d}\x\mathrm{d}t\right]\label{phoh}
\end{eqnarray} 
\end{theorem} 
\begin{proof}
We observe that energy inequalities separately for $(u^1,v^1,\theta^1)$ and $(u^2,v^2,\theta^2)$ implies 
$$
\int_{0}^{T}\int_{\Omega_1}|\nabla\theta|^2\mathrm{d}\x\mathrm{d}t\leq C_R.
$$
Then inequality \eqref{phoh} obviously follows from \eqref{as_key}. 
\end{proof}

\noindent{\it Proof of Theorem \ref{assmooth}:}
It follows from \eqref{phoh} that given $\varepsilon >0$ there exists $T>T_0(\varepsilon, B)$ such that for initial data $y_1,y_2\in B$ we have 
$$
||S_Ty_1-S_Ty_2||^2_{\mathcal{H}}\leq \varepsilon +\Psi_{\varepsilon,B,T}(y_1,y_2)
$$  
where 
\begin{eqnarray}\nonumber
&&\Psi_{\varepsilon,B,T}(y_1,y_2)= C_{B,\varepsilon,T}\left\{\Psi_{T}(u,v)+\int_{0}^{T}\int_{\Omega_1}|u|^2\mathrm{d}\x\mathrm{d}t+\int_{0}^{T}\int_{\Omega_2}|v|^2\mathrm{d}\x\mathrm{d}t\right\}+\\\nonumber
&&\hspace{70 pt}+C_{B,\varepsilon,T}\left\{\int_{0}^{T}\int_{\Omega}|\Delta^{-1}_{D}\left\{\rho_1u_t,\rho_2v_t\right\}|^2\mathrm{d}\x\mathrm{d}t\right\}
\end{eqnarray}
where $\Psi_T(u,v)$ is given by \eqref{psipsi} and $u$ and $v$ is the first and second component of $y_1-y_2$. 

To verify \eqref{th_Psi} we consider $\left\{y^{n}_{0}\right\}$ from $B$ and $(u^n(t),v^n(t),u_{t}^{n}(t),v^{n}_{t}(t),\theta^n(t))\equiv S_ty^{n}_{0}$. Then 
$$\begin{array}{l}
\left\{u^n(t),v^n(t)\right\}\;-\;\mathrm{is\;bounded\;in}\;L^{\infty}(0,T;H^2_0(\Omega))\\
\left\{u^n_t(t),v^n_t(t)\right\}\;-\;\mathrm{is\;bounded\;in}\;L^{\infty}(0,T;L^2(\Omega))\\ 
\end{array}
$$
We also claim that 
$$\left\{\rho_1u^n_{tt}(t),\rho_2v^n_{tt}(t)\right\}\;-\;\mathrm{is\;bounded\;in}\;L^{\infty}(0,T;H^{-2}).$$
Indeed, 
\begin{eqnarray}
\nonumber 
&&\int_{\Omega}|\Delta^{-1}_{D}\left\{\rho_1u^n_{tt},\rho_2v^n_{tt}\right\}|^2\mathrm{d}\x=-\int_{\Omega}\left\{\beta_1\Delta u^n+\mu \theta^n ,\beta_2\Delta v^n\right\}\Delta^{-1}_{D}\left\{\rho_1u_{tt}^n,\rho_2v_{tt}^n\right\}\mathrm{d}\x-\\\nonumber
&&-\int_{\Omega}\left\{F_1(u^n,v^n),F_2(u^n,v^n)\right\} \Delta^{-2}_{D}\left\{\rho_1u_{tt}^{n},\rho_2v_{tt}^{n}\right\}\mathrm{d}\x\leq C+ \frac{1}{2} \int_{\Omega}|\Delta^{-1}_{D}\left\{\rho_1u_{tt}^{n},\rho_2v_{tt}^n\right\}|^2\mathrm{d}\x
\end{eqnarray}
where $C>0$ does not depend on $n$. 
To prove subsequential limit \eqref{th_Psi} in $\Psi_T(u,v)$ we need the same arguments as in Step 2 of the proof of Proposition 3.36 in \cite{Manuscript} with $\Pi(u,v)$ definded by \eqref{Pi_1} (this is the place where we need conditions \eqref{Pi_4} and \eqref{Pi_5}). To pass to limit in $$\int_{0}^{T}\int_{\Omega_1}|u^n|^2\mathrm{d}\x\mathrm{d}t+\int_{0}^{T}\int_{\Omega_2}|v^n|^2\mathrm{d}\x\mathrm{d}t+\int_{0}^{T}\int_{\Omega}|\Delta^{-1}_{D}\left\{\rho_1u^n_t,\rho_2v^n_t\right\}|^2\mathrm{d}\x\mathrm{d}t   $$
we need Ascoli theorem (see \cite[\S8, Corollary 4]{Simon}). 

We apply Proposition \ref{fromManu} and we conclude the proof of asymptotic smoothness (and, thus, Theorem \ref {assmooth}).

\subsection {Attractors}
In this subsection first we obtain common result about local attractors (we denote them by $\mathcal{A}_R$) which follows directly from asymptotic smoothness (Theorem \ref{local_1}). Then we formulate two results devoted to their properties: Theorem \ref{local_2} states that the attractors consist of strong solutions (this result is for problem A) and Theorem \ref{attr_local} states smoothness and finite dimensionality of the attractors (this is for problem B). Theorem \ref{4m} about global attractors completes the Subsection 4.3. 

We note that if we consider the family of sets $\mathcal{E}_R:=\left\{x\in\mathcal{H}|\mathcal{E}(x)\leq R\right\}$, $R>0$ where $\mathcal{E}$ is defined in \eqref{Lyap}, we may consider $(\mathcal{E}_R, S_t)$ as a dynamical system, since $\mathcal{E}_R$ is a positively invariant set. Asymptotic smoothnes, result of the previous subsection, implies that $(\mathcal{E}_R, S_t)$ possesses a compact attractor $\mathcal{A}_R$ (actually, asymptotic smoothness implies existence of a compact attracting set, how to choose it strictly invariant we refer to \cite[Lemma 2.9]{Raugel}). 
\begin{theorem}\label{local_1}
Let conditions of Theorem \ref{assmooth} hold. Then for all $R>0$ the dynamical system $(\mathcal{E}_R, S_t)$ possesses a compact attractor $\mathcal{A}_R$.  
\end{theorem} 

To proceed we use idea of stabilizability estimate. Having been obtained it helps to prove that the attractor is smooth (see inequality \eqref{boundA_R}) and of finite fractal dimension. To prove \eqref{boundA_R} using stabilizability estimate one needs to repeat all steps from the proof of the \cite[Theorem 4.17]{Manuscript}, to prove finite dimensionality - \cite[Theorem 4.4]{Manuscript}, so we just formulate and prove stabilizability estimate in the proof of Theorem \ref{attr_local}. For other applications (exponential attractors, determining functionals, rate of convergence to equilibria and etc.) of the stabilizability estimate we refer to \cite{Manuscript} (there are many other works where these ideas are applied, see, e.g.,  \cite{ChuiBucci, Tynd, ChuiLasie, Karman, ChTynd,khan, Potomkin}).

Let us consider two weak solutions $(u^1(t),v^1(t),\theta^1(t))$ and $(u^2(t),v^2(t),\theta^2(t))$ of original problem \eqref{1}-\eqref{6} which satisfy \eqref{51}. Stabilizability estimate states that there exists such $C_R$ and $\omega_R$ (these constants do not depend on choise of $(u^i(t),v^i(t),\theta^i(t))$) that for all $t\geq 0$ 
\begin{equation}\label{stab_ineq}
E(t)\leq C_Re^{-\omega_R t}E(0)+C_R\left\{\int_{0}^{t}\int_{\Omega}|\left\{u,v\right\}|^2+|\Delta_{D}^{-1}\left\{\rho_1u_t,\rho_2v_t\right\}|^2\mathrm{d}\x\mathrm{d}\tau\right\}.
\end{equation}         
Here $E(t)$ is the same as in \eqref{energy} and again $u(t)=u^1(t)-u^2(t)$, $v(t)=v^1(t)-v^2(t)$ and $\theta(t)=\theta^1(t)-\theta^2(t)$.

First, we note that in critical case of nonlinearity (for example, probelm A) we cannot obtain this estimate because using the standard procedure (see \cite{Manuscript}) we need to estimate $\int_{0}^{\infty}||\left\{u_t,v_t\right\}||^2$ what is easy when the damping is mechanic (e.g., when resistance of media is accounted in monotonne functions $...+g(u_t)...$ and $...+g(v_t)...$ in left hand side of \eqref{1} and \eqref{3}). In purely thermoelastic problem this obstacle was overpassed (see \cite[Lemma 5.7]{ChuiLasie}, also we refer to \cite[Lemma 6.2]{ChuiBucci} and \cite[Lemma 5.1]{Potomkin}), but this method is not applicable here directly since it is complicated to estimate the component $v_t$ by $\theta$. However, in the case of problem A we can obtain that the attractor is smooth (consists of strong solutions). To estimate $||\left\{u_t,v_t\right\}||$ we use the structure $\mathcal{A}_R=\mathcal{M}^{u}(\mathcal{N})$ ($\mathcal{N}=\left\{x\in\mathcal{H}|S_tx=x, \forall t>0\right\}$ is the set of stationary points) which implies that for all trajectory $\gamma\subset\mathcal{A}_R$ and $\varepsilon>0$ there exists $T_{\gamma,\varepsilon}$ that $||\left\{u_t,v_t\right\}||< \varepsilon,$ $-\infty<t< T_{\gamma,\varepsilon}$. This idea is borrowed from \cite{Tynd} (see also \cite{Manuscript}). We refer to \cite[Proposition 5.4, Step 1]{Tynd} how to obtain the result of the following theorem from the inequality \eqref{as_key}.         
\begin{theorem} \label{local_2}
Let conditions of the Theorem \ref{assmooth} be in force. Assume also that $\mathcal{A}_R=\mathcal{M}^{u}(\mathcal{N}\cap\mathcal{E}_R)$. Then there holds $\mathcal{A}_R\subset H^4(\Omega_1)\times H^4(\Omega_2)\times H^2(\Omega_1)\times H^2(\Omega_2)\times H^2(\Omega_1)$.   
\end{theorem}

In the next theorem one more condition on nonlinearity is required. The norm $||\cdot||_{s}$ of space $H^s$ was introduced in Subsection 4.2. Note also that $H^1=H^1_0(\Omega)$.   
\begin{theorem}\label{attr_local}
Let all conditions of the Theorem \ref{assmooth} hold and for all $||\left\{u^i,v^i\right\}||_2<r$ there exists such $\delta>0$ that
\begin{equation} \nonumber
||\left\{F_1(u^1,v^1)-F_1(u^2,v^2),F_2(u^1,v^1)-F_2(u^2,v^2)\right\}||_{\delta}\leq c(r)
||\left\{u^1-u^2,v^1-v^2\right\}||_{2}.
\end{equation}
Consider any $R>0$, then there exists such $C_R>0$ that for any complete trajectory $\gamma=\left\{(u(t),v(t),u_t(t),v_t(t),\theta(t))|t\in\mathbb{R}\right\}\subset \mathcal{A}_{R}$ and for any $t\in \mathbb {R}$ there holds 
\begin{equation}\label{boundA_R}
\int_{\Omega}|\Delta^2 \left\{u,v\right\}|^2+|\Delta\left\{u_t,v_t\right\}|^2+|\left\{u_{tt},v_{tt}\right\}|^2\mathrm{d}\x+
\int_{\Omega_1}|\theta_t|^2+|\Delta\theta|^2\mathrm{d}\x\leq C_R.
\end{equation} 
Besides, $\mathcal{A}_R$ is the set of finite fractal dimension $d_R$.  
\end{theorem} 
\begin{proof}
We prove \eqref{stab_ineq}. After that to establish the result one should follow the line of Theorems 4.4 and 4.17 in \cite{Manuscript}.

Note that \eqref{51} is satisfied ($R$ from \eqref{51} may differ from $R$ here). Then we may use inequality \eqref{as_key}. 

We estimate $\Psi_T(u,v)$. 
\begin{eqnarray}
\nonumber &&
\int_{\Omega_1}g_1(t)u_t\mathrm{d}\x+\int_{\Omega_2}g_2(t)v_t\mathrm{d}\x=\int_{\Omega}\Delta^{\delta}_{D}\left\{\frac{g_1(t)}{\rho_1},\frac{g_2(t)}{\rho_2}\right\}\Delta^{-\delta}_{D}\left\{\rho_1u_t,\rho_2v_t\right\}\mathrm{d}\x\leq\\
\label{starDavid} &&\hspace{70 pt} \leq \varepsilon E(t)+C_{R,\varepsilon}\int_{\Omega}|\left\{u,v\right\}|^2+|\Delta^{-1}_{D}\left\{\rho_1u_t,\rho_2v_t\right\}|^2\mathrm{d}\x\;\;\;\;\forall \varepsilon>0.  
\end{eqnarray}  
Thus 
\begin{equation}
\nonumber
\Psi_T(u,v)\leq \varepsilon (1+T)\int_{0}^{T}E(t)\mathrm{d}t+C_{R,\varepsilon}(1+T)\int_{0}^{T}\int_{\Omega}|\left\{u,v\right\}|^2+|\Delta^{-1}_{D}\left\{\rho_1u_t,\rho_2v_t\right\}|^2\mathrm{d}\x\mathrm{d}t.
\end{equation}
From energetical equality and \eqref{starDavid} we also have
\begin {eqnarray} \nonumber
&&\beta_0\int_{0}^{T}\int_{\Omega_1}|\nabla\theta|^2\mathrm{d}\x\mathrm{d}t\leq E(0)-E(T)+\varepsilon\int_{0}^{T}E(t)\mathrm{d}t+\\
\nonumber&&\hspace{70 pt} +C_{R,\varepsilon}\int_{0}^{T}\int_{\Omega}|\left\{u,v\right\}|^2+|\Delta^{-1}_{D}\left\{\rho_1u_t,\rho_2v_t\right\}|^2\mathrm{d}\x\mathrm{d}t.
\end{eqnarray} 
Choosing sufficiently small $\varepsilon$ we obtain from \eqref{as_key} 
\begin{equation}
\label{stab}
E(T)\leq \frac{C_{R}}{T}\left(E(0)+C_{R,T}\int_{0}^{T}\int_{\Omega}|\left\{u,v\right\}|^2+|\Delta^{-1}_{D}\left\{\rho_1u_t,\rho_2v_t\right\}|^2\mathrm{d}\x\mathrm{d}t\right).
\end{equation}
This inequality transforms into \eqref{stab_ineq} with the help of manipulations the same as in \cite[Remark 3.30]{Manuscript}.  
\end{proof}  

To prove the existence of a global attractor we apply Theorem \ref{31} with $\mathcal{B}=\mathcal{N}$ where $\mathcal{N}$ is the set of stationary points. For both problems A and B the set $\mathcal{N}$ is bounded. Thus we have to find appropriate Lyapunov function and prove that this function is constant (with respect to $t \in \mathbb R$  when weak solution is substituted) only when argument is a stationary point. We take Lyapunov function given by \eqref{Lyap}.

We formulate auxiliary result inspired by \cite[Lemma 2.3]{Avalos2} and \cite[Lemma 3.1]{mr}. 
Consider a problem in $\Omega_2$
\begin{equation}\label{pa1}
\rho_2w_{tt}+\beta_2\Delta^2 w- b(t)\Delta w=0
\end{equation}
with boundary conditions
$$
w=\frac{\partial w}{\partial \nu}=0, \;\;\x\in \Gamma_2\cup\Gamma_0
$$
We denote 
$$
E_w(t)=\frac{1}{2}\int_{\Omega_2}\rho_2|w_t|^2+\beta_2|\Delta w|^2\mathrm{d}\x.
$$
The function $b(t)\in C[0,T]$. 
\begin{lemma}
There holds 
\begin{list}{}{}
\item[{1.}] There exists $C>0$ such that for all $T>0$
\begin{equation}\label{D_just}
\int_{0}^{T}\int_{\Gamma_2\cup\Gamma_0}|\Delta w|^2\mathrm{d}\Gamma\mathrm{d}t\leq C \int_{0}^{T}E_w(t)\mathrm{d}t+C(E_w(0)+E_w(T)). 
\end{equation}
\item[{2.}] Let $m(\x)=\x-\x_0$ and assume that \eqref{geom2} holds. There exists $C>0$ such that for all $T>0$
\begin{equation}\label{mge}
\int_{0}^{T}E_w(t)\mathrm{d}t\leq C\int_{0}^{T}\int_{\Gamma_0}m\cdot\nu|\Delta w|^2\mathrm{d}\Gamma\mathrm{d}t+C(E_w(0)+E_w(T)). 
\end{equation}
\end{list}
\end{lemma}
\begin{proof}
Let us consider the vector field $h=(h_1,h_2)\in [C^2(\Omega_2)]^2$ such that $h(\x)=\nu, \forall \x \in \Gamma_2\cup \Gamma_0$. Then 
\begin{eqnarray}\nonumber
&&\frac{\mathrm{d}}{\mathrm{d}t}\int_{\Omega_2}\rho_2w_th\cdot w\mathrm{d}\x=
\\\nonumber&&\hspace{20 pt}=
\frac{\beta_2}{2}\int_{\Gamma_2\cup\Gamma_0}|\Delta w|^2\mathrm{d}\Gamma-
\frac{1}{2}\int_{\Omega_2}(\nabla\cdot h)(\rho_2|w_t|^2-\beta_2|\Delta w|^2-b(t)|\nabla w_t|^2)\mathrm{d}\x +\\\nonumber&&\hspace{40 pt}
-\int_{\Omega_2}\left(b(t)\frac{\partial h_k}{\partial x_j}\frac{\partial w}{\partial x_j}\frac{\partial w}{\partial x_k}-2\beta_2\Delta w\frac{\partial h_k}{\partial x_j}\frac{\partial ^2 w}{\partial x_k\partial x_j}\right)\mathrm{d}\x-\int_{\Omega_2}\beta_2 \Delta w \Delta h_k\frac{\partial w}{\partial x_k}\mathrm{d}\x.
\end{eqnarray}
Integrating over $[0,T]$, using Young`s inequality we obtain \eqref{D_just}.  
To prove second part of this Lemma we consider 
\begin{eqnarray}
\nonumber&&\frac{\mathrm{d}}{\mathrm{d}t}\int_{\Omega_2}\rho_2w_tm\cdot\nabla w\mathrm{d}\x=\\\nonumber&&\hspace{20 pt}=\frac{\beta_2}{2}
\int_{\Gamma_2\cup \Gamma_0}(m\cdot\nu)|\Delta w|^2\mathrm{d}\Gamma -\int_{\Omega_2}\left(\rho_2|w_t|^2+\beta_2|\Delta w|^2\right)\mathrm{d}\x 
\end{eqnarray}
From integrating over $[0,T]$ \eqref{mge} follows. 
\end{proof} 
There holds 
\begin{proposition} \label {graddd}
Let $x=(u_0,v_0,u_1,v_1,\theta_0)\in\mathcal{H}$.
Let conditions of Theorem \ref{assmooth} hold. Besides, we consider problem A and problem B with additional assumption $f_2\equiv0$. Then if $\mathcal{E}(S_tx)=\mathcal{E}(x)$ for all $t\in\mathbb R$ then $S_tx=x$ for all $t\in\mathbb R$.    
\end{proposition}
\begin{proof} 
From \eqref{Eniq}, Poincaret inequality and $\mathcal{E}(S_tx)=\mathcal{E}(x)$ we have that $\theta\equiv 0$. \\
To proceed we use variation definition of a weak solution. First we substitute $\phi_1=\phi_2=0$ and any $\phi_3\in C^1(0,T;H^1_D):\phi_3(0)=\phi_3(T)=0$ and obtain that 
$$
\int_{0}^{T}\int_{\Omega_1}\Delta u \phi_{3,t}\mathrm{d}\x\mathrm{d}t=0
$$  
and, hence, for arbitrary $\phi_3\in H^1_0(0,T;H^2_0(\Omega_1))$
$$
\int_{0}^{T}\int_{\Omega_1}u_t \Delta \phi_3 \mathrm{d}\x\mathrm{d}t=0
$$
and, then $u_t=0$.

Now, if we set $w(t)=v(t+h)-v(t)$ then $w(t)$ for arbitrary $h>0$ satisfies 
\begin{eqnarray}
&&\rho_2w_{tt}+\beta_2\Delta^2w=F_2(u,v(t))-F_2(u,v(t+h)),\;\mathrm{in}\;\x\in\Omega_2\label{ucp1}\\
&&w|_{\Gamma_0\cup \Gamma_2}=\frac{\partial w}{\partial\nu}|_{\Gamma_0\cup \Gamma_2}=\Delta w|_{\Gamma_0}=\frac{\partial\Delta w}{\partial\nu}|_{\Gamma_0}=0.\label{ucp2}
\end{eqnarray}

And our aim is to prove that $w\equiv0$. It is easy when $\frac{\partial F_2}{\partial v}=0$ (this is the case of problem B with $f_2=0$), since then \eqref{ucp1} and \eqref{ucp2} is a well-known unique continuation problem solved, e.g., in \cite{LT}. In this case we can use \eqref{mge} in the following way 
\begin {equation}\label{ettt}
\int_{-T_1}^{T_2}E_{w}(t)\mathrm{d}t\leq C(E_w(-T_1)+E_w(T_2)),\;\;T_1+T_2>0 
\end{equation}
where $E_w(t)=\int_{\Omega_2}\rho_2|w_t|^2+\beta_2|\Delta w|^2\mathrm{d}\x$. After that, tending $T_i\rightarrow +\infty$ and using boundness of $E_w(t)$ we obtain that $E_w(t)\in L^1(\mathbb{R})$ which in view that $E_w$ is a positive implies that there exist $t_n^{\pm}\rightarrow\pm\infty$ such that $E(t_n^{\pm})\rightarrow 0$ if $n\rightarrow\infty$. Finally, setting $T_2=t^{+}_{n}$ and $T_1=-t^{-}_n$ again tending $n\rightarrow\infty$ in \eqref{ettt} we have $E_w(t)\equiv 0$. 

We continue with problem A.  
    
Taking $\phi_1 \in  H^2_0(\Omega_1)$ and $\phi_2=0$ we obtain 
$$
\int_{\Omega_1}\beta_1\Delta u \Delta \phi_1\mathrm{d}\x=
M(u,v)\int_{\Omega_1}u \Delta \phi_1\mathrm{d}\x
$$
Then either $u\equiv 0$ or $M(u,v)$ does not depend on $t$. \\
{CASE I.} If $u\equiv0$ then $M(u,v)=M(v)$ and $v$ is a weak solution to the problem \eqref{pa1} with $b(t)=M(v)$  and, moreover, $\Delta v=0$ in $L^2((0,T)\times\Gamma_0)$. Also we have energy relation 
\begin{equation}\label{pae}
E_v(T)+\frac{1}{2}M^2(v(T))=E_v(0)+\frac{1}{2}M^2(v(0))
\end{equation}  
Let us use \eqref{mge}
\begin{equation}\label{aaaa1}
\int_{0}^{T}E_v(t)\mathrm{d}t\leq C(E_v(0)+E_v(T))
\end{equation}
In particular, we get that there exists a sequence $t^{+}_n\rightarrow\infty$ such that $E_v(t_n^{+})\rightarrow 0$ and, thus, $M^2(v(t^{+}_n))\rightarrow \Gamma^2$. From \eqref{pae} we get 
\begin{equation}\label{aaaa2}
E_v(t)+\frac{1}{2}M^2(v(t))=\frac{1}{2}\Gamma^2,\;\;\;\forall t\in \mathbb R
\end{equation}
and since 
\begin{equation}\nonumber
\int_{t}^{T}E_v(\tau)\mathrm{d}\tau\leq C(E_v(t)+E_v(T))
\end{equation}
for all $-\infty<t<T<+\infty$ we have $E_v(t)\equiv 0$.\\
{CASE II.}
If $u\neq 0$ in $L^2$ then $M(u,v)=M$ does not depend on $t$ and, particularly, $\int_{\Omega_2}|\nabla v|^2\mathrm{d}\x$ does not depend on $t$. Consider $h>0$ and $w(t)=v(t+h)-v(t)$ is a weak solution to the problem \eqref{pa1} with $b(t)=M$ and also $\Delta w=0$ in $L^2((0,T)\times\Gamma_0)$. We use \eqref{mge}
$$
\int_{0}^{T}E_w(t)\mathrm{d}t\leq C(E_w(0)+E_w(T))
$$ 
and the energy relation has the following form 
$$
E_w(T)+\frac{1}{2}M||\nabla w(T)||^2_{L^2}=E_w(0)+\frac{1}{2}M||\nabla w(0)||^2_{L^2} 
$$
If $M\geq0$ then in view that $||\nabla w(t)||^2_{L^2}\leq C E_w(t)$ we have 
$$\begin{array}{l}
\int_{0}^{T}E_w(t)+\frac{1}{2}M||\nabla w(t)||^2_{L^2}\mathrm{d}t=T\left(E_w(0)+\frac{1}{2}M||\nabla w(0)||^2_{L^2}\right)
\leq\\\hspace{20 pt}\leq C(2E_w(0)+\frac{1}{2}M(||\nabla w(0)||^2_{L^2}-||\nabla w(T)||^2_{L^2}))\leq(2+\frac{1}{2}M)CE_w(0)
\end{array}$$ 
for all $T>0$. Then $E_w(t)\equiv 0$.\\
If $-|\Gamma|\leq M <0$ then $||\nabla w(t)||^2_{L^2}\leq 4 \Gamma^2/ \gamma$ and 
$$
\int_{0}^{T}E_w(t)\mathrm{d}t\leq C \left( 2E_w(0)+\frac{1}{2}|M|(||\nabla w(T)||^2_{L^2}+||\nabla w(0)||^2_{L^2})\right)\leq 2C(E_w(0)+|M|\Gamma/\gamma)
$$
From energy identity we have $E_w(t)=\frac{1}{2}|M|||\nabla w (t)||^2\leq 4|M| \Gamma^2/ \gamma$ for all $t>0$ and $h>0$. Obviously this estimate is extended on $t\in\mathbb R$. Again we have the inequality 
\begin{equation}\nonumber
\int_{t}^{T}E_w(\tau)\mathrm{d}\tau\leq C(E_w(t)+E_w(T))
\end{equation}
for all $-\infty<t<T<+\infty$. Thus $E_w(t)\equiv 0$ and, hence, $v(t)=v(0)$.    
 \end{proof}

Now we may formulate main result devoted to attractor. 
\begin{theorem} \label{4m}
Let conditions of Theorem \ref{assmooth} hold. Then dynamical system $(\mathcal{H},S_t)$ corresponding to problem A possesses a compact global attractor $\mathcal {A}=\mathcal{M}^{u}(\mathcal{N})$ from $H^4(\Omega_1)\times H^4(\Omega_2)\times H^2(\Omega_1)\times H^2(\Omega_2)\times H^2(\Omega_1)$. The dynamical system $(\mathcal{H},S_t)$ corresponding to problem B with additional condition $f_2\equiv0$ possesses a compact global attractor $\mathcal {A}=\mathcal{M}^{u}(\mathcal{N})$ of finite fractal dimension and for any full trajectory $\gamma$ from the attractor \eqref{boundA_R} holds (of course, with $C$ instead of $C_R$).  
\end{theorem}
\begin{proof}
To establish the existence of a global attractor we apply Theorem \ref{31}. The statement about smoothness and finite dimensionality of attractor in prblem B follows from the observation that for some $R_0>0$ (and, moreover, for all $R>R_0$) dynamical systems $(\mathcal{H},S_t)$ and $(\mathcal{E}_{R_0},S_t)$ possess common compact attractor. Therefore we may use the corresponding statements of Theorem \ref {attr_local} for some large enough $R_0$.   \end{proof}

\begin{remark}
The main obstacle in proof of the existence of attractor is to establish that if $w(t)$ satisfies \eqref{ucp1} with overdefined boundary conditions \eqref{ucp2} then $w(t)\equiv 0$. We may rewrite 
\begin{eqnarray}
&&\rho_2w_{tt}+\beta_2\Delta^2w=B(t)w,\;\;
w|_{\Gamma_0\cup \Gamma_2}=\frac{\partial w}{\partial\nu}|_{\Gamma_0\cup \Gamma_2}=\Delta w|_{\Gamma_0}=\frac{\partial\Delta w}{\partial\nu}|_{\Gamma_0}=0.\label{ucp3}
\end{eqnarray}  
where $B(t)=\int_{0}^{1}\Pi''(u,v+\lambda w)\mathrm{d}\lambda$ is a linear operator such that $||B(t)w||\leq c(r)||w||_2$ for all $||\left\{0,w\right\}||_2+||\left\{u,v\right\}||_2\leq r$. For example, in case of problem B we have 
$
B(t)w=\int_{0}^{1}f'_2(v+\lambda w)\mathrm{d}\lambda\cdot w
$. If one proves that the equalities in \eqref{ucp3} with nonlinearity as in problem B or C implies that $w(t)\equiv0$ then it would give the existense of a global attractor $\mathcal {A}=\mathcal{M}^{u}(\mathcal{N})$ in the corresponding problem.   
\end{remark}

\end{document}